\documentclass[11pt]{article}

\usepackage{amsmath}
\usepackage[psamsfonts]{amssymb}
\usepackage{amsmath}
\usepackage{amsthm}
\usepackage{amsfonts}
\usepackage{color}
\usepackage{colortbl}
\usepackage{algorithmic}
\usepackage{graphicx}
\usepackage{fullpage}
\usepackage{subfigure}
\usepackage{url,hyperref}

\newtheorem{theorem}{Theorem}[section]
\newtheorem{lemma}[theorem]{Lemma}
\newtheorem{proposition}[theorem]{Proposition}
\newtheorem{corollary}[theorem]{Corollary}

\newcommand{\be}{\begin{eqnarray}}
\newcommand{\ee}{\end{eqnarray}}

\def\mc{\mathcal}
\def\mbf{\mathbf}
\def\argmin{\mathop{\mathrm{argmin}}}

\title{Generalized Opinion Dynamics from Local Optimization Rules}
\author{
Avhishek Chatterjee,  Anand~D.~Sarwate, Sriram~Vishwanath
\thanks{A. Chatterjee and S. Vishwanath are with the Department of Electrical and Computer Engineering at The University of Texas at Austin, 1616 Guadalupe Street, Austin, TX 78701 (e-mail: \texttt{avhishek@utexas.edu}, \texttt{sriram@ece.utexas.edu}).  A.D. Sarwate is with the Department of Electrical and Computer Engineering at Rutgers, The State University of New Jersey, 94 Brett Road, Piscataway, NJ 08854 (e-mail: \texttt{asarwate@ece.rutgers.edu}).}%
\thanks{The work of A.D. Sarwate is supported in part by the National Science Foundation under award CCF-1440033.  The work of Avhishek Chatterjee and SriramVishwanath is supported by ARO grant numbers W911NF-13-1-0269 and W911NF-11-1-0258}
}
\date{\today}

\begin{document}

\maketitle

\begin{abstract}
We study generalizations of the  Hegselmann-Krause (HK) model for opinion dynamics, incorporating features and parameters that are natural components of observed social systems.  The first generalization is one where  the strength of influence depends on the distance of the agents' opinions.  Under this setup, we identify conditions under which the opinions converge in finite time, and provide a qualitative characterization of the equilibrium.  We interpret the HK model opinion update rule as a quadratic cost-minimization rule. This enables a second generalization: a family of update rules which possess different equilibrium properties.  Subsequently, we investigate models in which a external force can behave strategically to modulate/influence  user updates.  We consider cases where this external force can introduce additional agents  and cases where they can modify the cost structures for other agents.  We describe and analyze some strategies through which such modulation may be possible in an order-optimal manner. Our simulations demonstrate that generalized dynamics differ qualitatively and quantitatively from traditional HK dynamics.
\end{abstract}

\section{Introduction}
Multiple disciplines have studied the dynamics of agents in distributed and networked systems. Whether studying the swarming/flocking behavior of animals, messaging among individuals and institutions in online social networks, or interactions among members in a village community, these studies typically build and rely on a {\em model} that captures interactions among agents. These models are diverse across disciplines, but they share some common ground - for instance, typically, individual agents in the system incorporate information they gain from their neighbors and take actions based on the information gained. Such commonalities between diverse sets of scenarios involve interacting agents enables us to develop and study mathematical models that can apply to several disciplines.

In this paper, we take a renewed and closer look at such a cross-cutting mathematical model called the Hegselmann-Krause (HK) model of opinion formation~\cite{hegselmannKrause02}. In the HK model, agent opinions are modeled as a point in a metric space and an agent can observe the opinions of other agents whose opinions are within a certain radius $\epsilon$ of their own opinion.  Each agent then synchronously updates its opinion by replacing it with the average of its neighbors' opinions.  This process continues either ad nauseam or until the system reaches an equilibrium.  In the case where agents' opinions can be modeled as points within $\mathbb{R}$, this rule is known to converge to an equilibrium, wherein each agent's opinion lies in a finite set of points.

The HK dynamics capture  interactions where
agents with a bounded difference in opinions - called neighbors (for agent $i$, $\mc{N}_i$ would be considered neighbors) exchange opinions and interact with one another. Currently, HK dynamics treats opinions within an agent's neighborhood evenly, and does not differentiate between them when modeling their influence on the particular agent. In general, however, the impact that neighboring opinions may have on a particular agents could be widely varying, and a more general model is needed that allows for variations in the manner in which an agent is affected by its neighbors. To this end, our first generalization of HK dynamics is to enable each node to take a weighted average of its neighbors' opinions. The particular model we study in detail is one in which the weights are functions of the differences in opinions between the agent and its neighbor. In this setting, we show that our modified dynamics also converge in finite time under certain mild assumptions; moreover, we provide a qualitative characterization of the equilibrium. 

\begin{figure}[t]
\begin{center}
\includegraphics[width=\columnwidth]{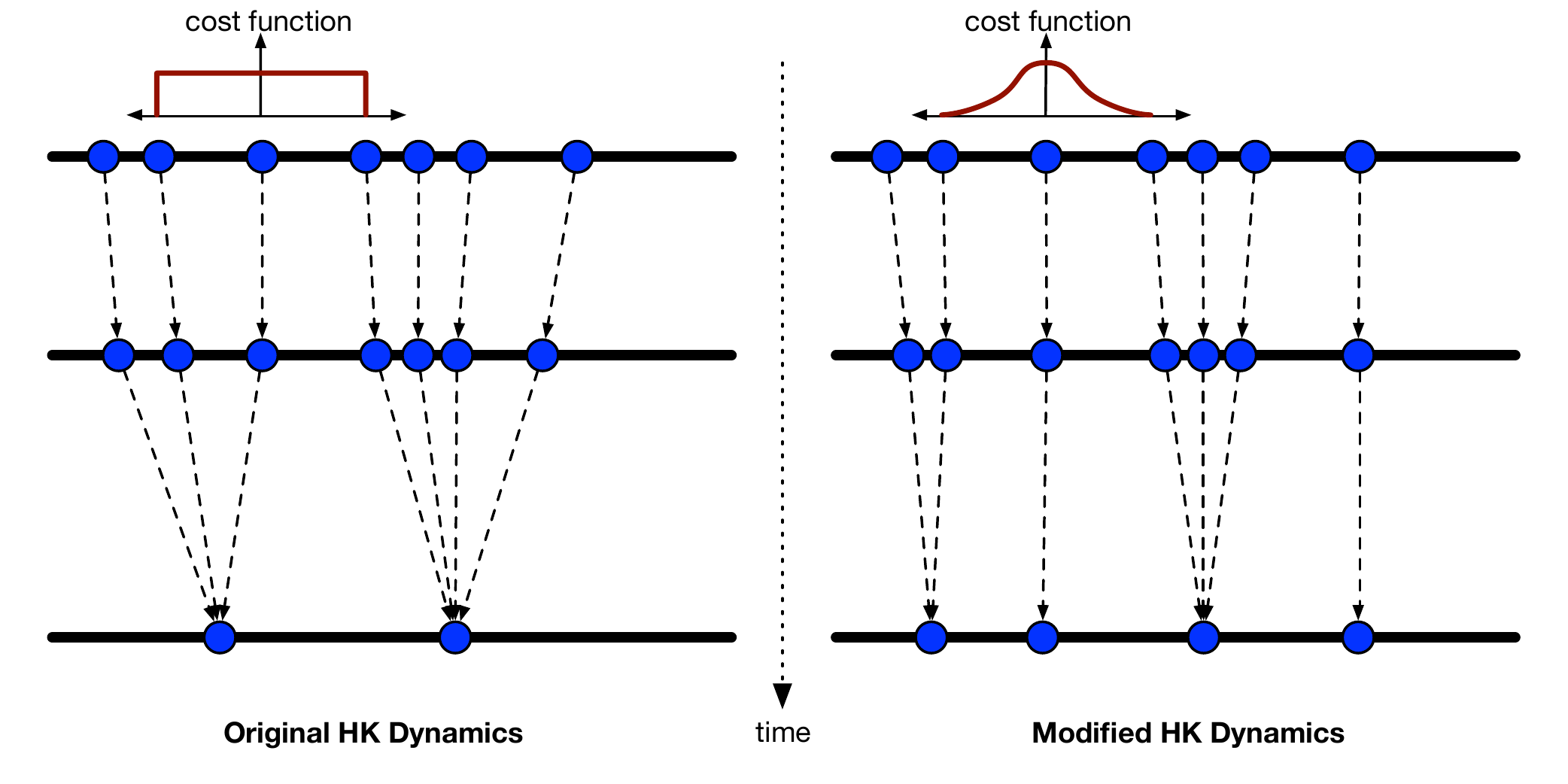}
\caption{Qualitative differences in the equilibrium behavior between the original and modified HK dynamics.}
\label{fig:cartoon}
\end{center}
\end{figure}

We can reinterpret the conventional HK update rule as an optimization problem: each agent is minimizing a disagreement function that depends on its own opinion and those of its neighbors.  Given such an reinterpretation of the HK update rule, we note that the traditional HK dynamics correspond to a case where agents act \textit{greedily} to minimize a quadratic cost function.  In this paper, we generalize the cost function structure to other convex cost functions.  In this more general setting, we provide analytic conditions on the disagreement functions that imply asymptotic convergence of the dynamics.  In general, changing the cost function will alter the speed convergence as well as the properties of the equilibrium.  For example, consider the situation as illustrated in Figure \ref{fig:cartoon}, where the cost decays with distance. This corresponds to a situation where an agent is less sensitive to larger disagreements, and as a result,  opinions within the population may coalesce into (possibly multiple) equilibrium points.  

Finally, given a model for opinion formation and dynamics, it is only natural to ask if external influence on the system may be possible, driving the entire system towards a desired point of operation. Such an external influencer is designed to model inputs to the system, including election campaigns, leaders and opinion-shapers that influence the agents within the system but are not themselves influenced by opinions within the system.  By formulating the HK dynamics as the solution of a local greedy optimization framework, we incorporate scenarios where external forces can endeavor to influence agents within the system.  In the first scenario, an external force influences the agents by introducing a fake external agent at each time to influence the updates other agents. In the second scenario, this influencer offers incentives to agents in order to effect changes in opinions. Effectively, this alters the cost structure for  disagreement minimization within the optimization formulation.  While exactly optimal solutions are difficult to obtain in these settings, we identify greedy strategies that perform well (in an order-optimal manner).

\subsection{Related Work}

Models for opinion dynamics have been studied in a variety of fields, including sociology, physics, computer science, and engineering.
The wide range of models~\cite{castellano2009statistical} can be divided into two broad approaches: dynamical systems and agent-based modeling. Agent-based models in which each social element or entity acts based on actions or positions of others have been studied extensively.  Agents take actions to pursue an objective such as social welfare maximization, individual benefit, or learning~\cite{acemogluLearning11,niyogiLanguage08,constConsNedic10,jacksonLearning10}.

Bounded-confidence opinion dynamics~\cite{hegselmannKrause02,deffuantWeisbuch00} have been proposed as agent-based dynamics for modeling opinion formation.  Unlike previous models
for social dynamics where interactions among social agents are governed by an underlying graph (fixed or
random, but independent of opinions), bounded confidence models make interactions opinion-dependent. In the Hegselmann-Krause model~\cite{hegselmannKrause02} every agent updates its opinion (modeled as a scalar) by
averaging all other opinions within a certain distance (threshold) from its own. 
This dynamics has been studied
extensively via numerical techniques as well as analytically. Tsitsiklis et al.~\cite{tsitsiklisHK09} 
proved that this dynamics converges in finite time and provided upper and lower bounds 
for the convergence times in terms of number of agents.  Bhattacharyya et al.~\cite{bhattacharyya2013convergence} and Nedic and Touri~\cite{nedic2012multi} studied multi-dimensional (vector) Hegselmann-Krause dynamics. 

Multiple variations of this dynamics and their evolutions have been studied, e.g., effect of 
different initial conditions~\cite{jacobmeier2006focusing}, noise in the updates~\cite{carro2013role,pineda2013noisy}, 
heterogeneous thresholds among agents~\cite{weisbuch2002meet}, or mediating interactions using an underlying social graph~\cite{deffuant2002can}; Lorenz~\cite{lorenz2007continuous} presents a survey of this line of work. 
Lorenz~\cite{lorenz2007repeated} also proposed a weighted Hegselmann-Krause dynamics with opinion dependent
weights, and Hendrickz~\cite{hendrickx2008order} studied conditions for order-preservation of opinions
in this dynamics.   In this work we consider Hegselmann-Krause dynamics with opinion-dependent weighted updates and study convergence
time of this dynamics.

In this paper we interpret the Hegselmann-Krause dynamics in terms of utility optimization. There is a significant body of work on  utility-maximization in multi-agent settings. Huang et al.~\cite{huang2007large} considers a large population linear Quadratic Gaussian (LQG) problem where agents interact via coupled cost. There are other works involving control theoretic, game theoretic and optimization based approaches. Cao et al.~\cite{cao2013overview} presents a survey of recent developments.  Here we study an engineering problem where the society follows Hegselmann-Krause dynamics and a strategic entity wants to form opinions in its favor. This problem of
opinion modulation/modification is analogous to marketing and political campaigns.

\section{Generalizing  HK dynamics} 

\paragraph*{Notation.}  For a positive integer $n$ let $[n] = \{1,2, \ldots, n\}$.  Vectors will typically be denoted in boldface, and sets in calligraphic script.  We use the standard ``big-$O$'' and ``big-$\Omega$'' notation for expressing upper and lower bounds.

There are $n$ agents in the system, indexed by $[n]$. The system evolution is discrete-time. At a time $t \in \mathbb{Z}_+$, each agent $i$ has an opinion $x_i(t)$ taking values in $\mathbb{R}$.  We write the opinions of the agents collectively by the $n$ dimensional vector $\mathbf{x}(t)=(x_1(t), x_2(t), \cdots x_n(t))$. At time $0$ the agents have initial opinion $\mathbf{x}(0)$.  The original Hegselmann-Krause dynamics are formally given as follows:
\begin{align}
\mathcal{N}_i(t) &= \{j: |x_j(t)-x_i(t)|\leq \gamma\}, \label{eq:HKnbr} \\
x_i(t+1) &=\frac{1}{| \mathcal{N}_i |} \sum_{j \in \mathcal{N}_i(t)} x_j(t) \nonumber \\
	&= x_i(t)+\frac{1}{|\mathcal{N}_i|} \sum_{j \in \mathcal{N}_i(t)} (x_j(t)-x_i(t)).
	\label{eq:HKupdateRewrite}
\end{align}
In HK dynamics, each agent determines those agents whose opinions are within a distance $\gamma$ of theirs to form a \textit{neighbor set} $\mathcal{N}_i(t)$ at time $t$ given by \eqref{eq:HKnbr}; we adopt the convention that $i \in \mathcal{N}_i(t)$.  All agents then simultaneously update their opinions by adding the average of the differences between their opinions and those of their neighbors.  Note that in this formulation, each agent gives equal weight to all other agents that are in the neighbor set.

We say an opinion dynamics process \textit{converges} if $x_i(t)$ is a converging sequence for each $i$. It \textit{converges in finite time} if 
there exists a finite $T$ s.t.~the set $x_i(t+1) = x_i(t)$ for all $i \in [n]$ and all $t \ge T$.  The original Hegselmann-Krause dynamics converges in finite time \cite{tsitsiklisHK09, nedic2012multi}; 
more specifically its convergence time is lower bounded as $\Omega(n)$ and upper bounded as $O(n^2)$. %

In this section we generalize the Hegselmann-Krause dynamics in an attempt to capture other aspects of social interactions. Towards this we take two different approaches that pertain to the real social systems. In the first approach, we understand that in real-world social systems, an individual is typically influenced by those whose opinions are `close' to theirs, although they may not value all such opinions equally. The HK dynamics use a binary model of closeness; all agents in $\mc{N}_i(t)$ have equal influence.  We first generalize HK dynamics by introducing non-uniform and distance-dependent weighting of the opinions of other agents. Subsequently, we interpret the agent update rules as an optimization problem and propose to generalize the opinion updation process by considering different optimization objectives.

\subsection{Non-uniform weights and its convergence}

We generalize the update rule \eqref{eq:HKupdateRewrite} to reflect the observation that
in real life the importance (or weight) an agent gives to another
agent's opinions is not a $0-1$ function of their difference in opinion; instead, the 
weight gradually decreases with increase in difference.  We correspondingly modify the update rule \eqref{eq:HKupdateRewrite}:
\begin{align}
x_i(t+1) &=x_i(t)+\frac{1}{|\mathcal{N}_i|} \sum_{j \in \mathcal{N}_i(t)} f\left(|x_j(t)-x_i(t)|\right) %
	(x_j(t)-x_i(t)),\label{eq:GenHKupdate}
\end{align}
where $f : \mathbb{R}_+ \to [0,1]$ is non-increasing and $f(0) = 1$.
The function $f$ and the threshold $\gamma$ depend on the level of interactions in a society; some 
societies may de-value differing opinions more strongly than others. It may also depend on the particular 
issue regarding which the opinion is evolving and other socio-economic and political factors. 

Our first question is how introducing $f$ affects the evolution of the dynamics: do they converge to an equilibrium, and if so, what is the convergence time?
For the original HK dynamics, a theoretical guarantee on convergence as well as time to convergence is
presented by Blondel et. al. \cite{tsitsiklisHK09}. Later Touri and Nedic~\cite{nedic2012multi} improved the upper-bound on convergence time
from $n^3$ to $n^2$ using advanced techniques from matrix-valued stochastic processes. These techniques do not
directly extend to general HK dynamics with non-uniform weights. We analyze the general HK dynamics using a
simpler analytical approach which also applies to original HK dynamics.

We characterize our generalized real-valued dynamics by tracking the
maximum and minimum of opinions; in order for the opinions to converge, these two must converge.
The following lemma characterizes the sequence of the range of opinions over time, and is analogous to Proposition 2 of Blondel et al.~\cite{tsitsiklisHK09}.

\begin{lemma}
The sequences $M(t)=\{\max_k x_k(t)\}$ and $m(t)=\{\min_k x_k(t)\}$ are non-decreasing and non-increasing respectively.
\label{lem:1andNchar}
\end{lemma}

\begin{proof}
It is clear from the update equation that the drift for $m(t)$ have all positive terms and that for
$M(t)$ have all negative terms respectively. Hence the Lemma follows.
\end{proof}

The following proposition follows from Lemma \ref{lem:1andNchar}.

\begin{proposition}
The generalized Hegselmann-Krause dynamics \eqref{eq:GenHKupdate} converges.  That is,
for any $\epsilon>0$ there exists a $t_0$ s.t. $\forall t\ge t_0, 1\le i \le n$, $|x_i(t)-x_i(t_0)|<\epsilon$.
\label{prop:conv}
\end{proposition}

\begin{proof}
Note that $m(t)\leq M(t)$ and also $M(t)\leq M(0)$, as $m(t)$ (resp.~$M(t)$) is non-decreasing 
(resp.~non-increasing). Hence $m(t)$ (resp.~$M(t)$) converges. When $m(t)$ converges to a point $m(\infty)$
there is a set of agents $\mc{A}$ converging to $m(\infty)$ and all other agents must have opinions greater than $m(\infty) + \gamma$.   Similarly, when $M(t)$ converges to a point $M(\infty)$
there is a set of agents $\mc{A}$ with opinions converging to $M(\infty)$ and all other agents must have opinions smaller than $M(\infty) - \gamma$. 

Thus there exists an $\epsilon>0$ such that there is no agent within $\gamma+\epsilon$ of $m(\infty)$ converges. Let the agents that do not converge to $m(\infty)$ be the set
$S$. Then we can say that $\lim \inf_{t \to \infty} \min_{i \in S} x_i(t)$ is at least
$m(\infty) + \gamma + \epsilon$. 
Therefore there exists a $t_0$ such that for $t>t_0$, $m(t)< m(\infty) + \frac{\epsilon}{2}$ and
	\[
	\inf_{t>t_0} \min_{i \in S_1} x_i(t) > \lim \inf_{t \to \infty} \min_{i \in S_1} x_i(t) - \frac{\epsilon}{2}.
	\]
This fact, together with the bounds derived above imply that for $t>t_0$,
	\[
	m(t) < \min_{i \in S} x_i(t)+\gamma.
	\]
Thus for $t>t_0$ the agents $S$ evolve according to dynamics independent of those converging to $m(\infty)$. Hence this can be treated as a HK dynamics staring with initial opinions
$x_i(t_0)$ which has a minimum opinion $m_S(t)$. Similarly we can show that $m_{S_1}(t)$ converges and
there exists $S_2\subset S_1$ satisfying similar conditions. Thus we can continue to split the dynamics 
into parts $S_1, S_2, \cdots$ until for some $k$ $S_k=\empty$. Note that $k\leq n$. For each of this subset
$S_j$ with $j<k$, agents in $S_{j+1}\backslash S_j$ converges to $\lim_{t \to \infty} m_{S_j}(t)$. Hence
this shows that overall dynamics converges.
\end{proof}

After proving the convergence of the generalized Hegselmann-Krause dynamics, we are interested in characterizing the convergence time.
Tsitsiklis et al.~\cite{tsitsiklisHK09} presents lower and upper bounds of 
$\Omega(n)$ and $O(n^3)$, the latter of which was improved by Touri~\cite{nedic2012multi} to $O(n^2)$. These results imply that there exist constants $c,c'>0$ such that  for any $\tau>0$, and for all $t\ge c n^2$,
$x_i(t)=x_i(t+\tau)$ for $1 \le i \le n$ and for any $t<c'n$ there exists $i$ such that $x_i(t+1) \neq x_i(t)$.
So far we have proved only asymptotic convergence of the general HK dynamics. We can also prove finite time convergence
of the general HK dynamics under some mild assumptions on the function $f$. 

To bound the convergence time, we make the following assumptions:
\begin{enumerate}
\item[(i)] There exists a positive constant 
$\epsilon < \gamma$ such that $f(x)=f(0)=1$ for $x \le \epsilon$.
\item[(ii)] The function $f$ is strictly positive: $f(\gamma)>0$.
\item[(iii)] The product $xf(x)$ is non-decreasing in $x$ for $0\leq x\leq\gamma$. 
\end{enumerate}
Note that convergence of the dynamics is guaranteed for any general $f$. These assumptions are required for finite-time convergence.

Intuitively, the first assumption means that an 
agent gives same importance to all its very close neighbors or friends as it gives to itself.  The second assumption requires
that there exists a strict gap between a neighbor (or friend) and a non-neighbor; this assumption is widely prevalent in different kinds literature including social networks and graphical models~\cite{netrapalli2010greedy}, where two variables are considered to  be neighbors in a dependence graph 
only if they have a correlation more than a strictly positive number.
With the above two assumptions on the function $f$ we can prove finite time convergence for the generalized
HK dynamics. The following theorem summarizes our result on the convergence time.

\begin{theorem}
\label{thm:UBconvTime}
Consider the generalized HK dynamics under the update rule in \eqref{eq:GenHKupdate} where $f$ satisfies assumptions (i)--(iii) above.  Then the convergence time of the minimum opinion $m(t)$ is $\text{O}(n^2)$ and 
that of the generalized HK dynamics is $\text{O}(n^3)$.  That is, for some $k_1, k_2>0$
$m(t)$ and the generalized dynamics converges within a time $k_1 n^2$ and $k_2 n^3$ for all $n$
sufficiently large. Also,
the convergence time of the generalized dynamics is $\Omega(n)$.  That is, there exists a $k_3>0$
such that convergence time is no less than $k_3 n$ for all $n$ sufficiently large. 
\end{theorem}

To prove this theorem we need the following lemma, which states that the ordering of opinions in the general HK dynamics
remains unchanged.

\begin{lemma}
Suppose the initial opinions are ordered as
$x_1(0) \leq x_2(0) \leq \cdots \leq x_n(0)$ and the function $f$ be such
that $xf(x)$ is non-decreasing in $x$ for $0\leq x\leq\gamma$, then for 
all $t \in \mathbb{Z}_+$, $x_1(t) \leq x_2(t) \leq \cdots \leq x_n(t)$.
\label{lem:orderingGHK}
\end{lemma}
\begin{proof}
First note that change in opinion for any node $i$, $x_i(t+1)-x_i(t)$ can be written as the sum of positive and negative drifts, defined as follows:
	\begin{align*}
	\frac{1}{|\mathcal{N}_i|} \sum_{j:x_j(t) \in [x_i(t), x_i(t)+\gamma)} f(|x_j(t)-x_i(t)|) (x_j(t)-x_i(t)) \\
	\frac{1}{|\mathcal{N}_i|} \sum_{j:x_j(t) \in [x_i(t), x_i(t)-\gamma)} f(|x_j(t)-x_i(t)|) (x_j(t)-x_i(t)).
	\end{align*}

Consider any two agents $k$ and $k'$ with $x_{k'}(t)>x_k(t)$ and there are no other agents in between them. 
To prove the order preserving nature of the update it is sufficient to ensure that $x_{k'}(t+1)\geq x_k(t+1)$.  

Note that the right neighbors (neighbors with strictly higher opinion values) of $k$ are also right neighbors of $k'$ (except $k'$ itself) and
left neighbors (neighbors with strictly lesser opinion values) of $k'$ are also left neighbors of $k$ (except $k$ itself). We denote
right and left neighbor set of an agent $i$ by $\mathcal{RN}_i$ and $\mathcal{LN}_i$ respectively.

We want to show that the order of opinions cannot change.  Were the order to change at time $t$, the worst neighborhood condition is the case where
the neighbors of $k$ and $k'$ are same. This follows because the function $f$ is such that
$xf(x)$ is non-decreasing for $x\leq\gamma$.  Thus if $k'$ has an extra neighbor then it must be a right neighbor that gives an additional
positive drift where as if $k$ has an extra neighbor it must be left neighbor that gives an additional negative drift.

Let $N=|\mathcal{N}_k|=|\mathcal{N}_{k'}|$. For simplicity of notation we shall drop indexing by $t$ and consider the drift in opinion of agent $k$.  We have
\begin{align}
N (x_k(t+1)-x_k(t)) %
=&\sum_{j \in \mathcal{RN}_{k'}} f(x_j-x_k)(x_j-x_k)-\sum_{r \in \mathcal{LN}_{k}} f(x_k-x_r)(x_k-x_r) \nonumber \\
& +f(x_{k'}-x_k)(x_{k'}-x_k) \label{eq:writingDrift} \\
= &\sum_{j \in \mathcal{RN}_{k'}} f(x_j-x_k)(x_j-x_{k'})+\sum_{j \in \mathcal{RN}_{k'}} f(x_j-x_k)(x_{k'}-x_{k}) \nonumber \\
& - \sum_{l \in \mathcal{LN}_{k}} f(x_k-x_l)(x_{k'}-x_l) + \sum_{l \in \mathcal{LN}_{k}} f(x_k-x_l)(x_{k'}-x_{k}) \nonumber \\
&  +f(x_{k'}-x_k)(x_{k'}-x_k) \label{eq:k'minusk} \\
\leq &\sum_{j \in \mathcal{RN}_{k'}} f(x_j-x_{k'})(x_j-x_{k'})+\sum_{j \in \mathcal{RN}_{k'}} f(x_j-x_k)(x_{k'}-x_{k}) \nonumber \\
& - \sum_{l \in \mathcal{LN}_{k}} f(x_{k'}-x_l)(x_{k}-x_l) + \sum_{r \in \mathcal{LN}_{k}} f(x_{k'}-x_k)(x_{k'}-x_{k}) \nonumber \\
&  +f(x_{k'}-x_k)(x_{k'}-x_k) \label{eq:monoDecr} \\
= & N (x_{k'}(t+1)-x_{k'}(t)) + \sum_{j \in \mathcal{RN}_{k'}} f(x_j-x_k)(x_{k'}-x_{k}) \nonumber \\
&  \sum_{r \in \mathcal{LN}_{k}} f(x_k-x_r)(x_{k'}-x_{k}) + 2 f(x_{k'}-x_k)(x_{k'}-x_k) \label{eq:k'drift} \\
\leq & N (x_{k'}(t+1)-x_{k'}(t)) + N (x_{k'}-x_k) \label{eq:weight} 
\end{align}

The equality (\ref{eq:writingDrift}) follows from the update equation for the Hegselmann-Krause dynamics.
By writing $(x_j-x_k)=(x_j-x_{k'})-(x_{k'}-x_k)$ and $(x_{k'}-x_r)=(x_k-x_r)-(x_k-x_{k'})$ we obtain the
equality (\ref{eq:k'minusk}).
As $f$ is non-increasing function and $x_j-x_{k'} > x_j-x_k$ for $j \in \mathcal{RN}_{k'}$ and 
$x_{k'}-x_l>x_k-x_j$ for $l \in \mathcal{LN}_k$, the inequality (\ref{eq:monoDecr}) follows.  Finally, (\ref{eq:k'drift}) follows from the update equation for generalized HK dynamics. The last step (\ref{eq:weight})
follows from the fact that $f(x)\leq 1$ and $|\mathcal{LN}_k \cup \mathcal{RN}_{k'}|=N-2$.
\end{proof}

We can now prove Theorem \ref{thm:UBconvTime}.

\begin{proof}[Proof of Theorem \ref{thm:UBconvTime}]
To prove the upper bound first we derive an upper bound on 
the convergence time for the minimum value of the opinion. Because the opinions are ordered, we can find such an upper bound by lower-bounding the increase in opinion of agent $1$. 

Note that in the general HK dynamics if the opinion of an agent $i$ has converged
to $x^*_i$ by time $t^*$, this means for all time $t\ge t^*$, for any agent $j\neq i$ either
$x_j(t)=x^*_i$ or $|x_j(t)-x_i^*|>\gamma$. This is because if there is an agent 
within $\gamma$ neighborhood of agent $i$ then by the update rule of the dynamics and by the
assumptions on $f$, $x_i(t+1)-x_i(t) \neq 0$. %

Thus for proving the upper-bound if we can show that a subset of agents have converged by certain finite time, then we can
consider the remaining agents (outside this subset) as a new system and analyze convergence time of these agents. 
This allows us to inductively prove the final upper-bound.

Our goal is to upper bound the convergence time of $x_1(t)$, the minimum value of the agents (this is 
because ordering of opinions is preserved by the dynamics).  There are three cases to consider.

We could have (i) $\mathcal{N}_1(t)=\{1\}$, so agent 1 has converged, (ii) $|\mathcal{N}_1(t)| \ge 2$ and there exists an agent $k$ with $x_k(t) \in (x_1(t)+\epsilon, \gamma]$ and (iii)  $|\mathcal{N}_1(t)| \ge 2$ but there exists no agent $k$ with $x_k(t) \in (x_1(t)+\epsilon, x_1(t)+\gamma]$.

Case (i) : In this case there is nothing to prove.

Case (ii): As $x_1(t) \leq x_k(t)$ for all $n\geq k\geq 2$, from the update equation \eqref{eq:GenHKupdate} of the generalized HK dynamics 
it is clear that the increase in the opinion value of the agent $1$ is given by 
	\[
	\frac{1}{|\mathcal{N}_1|} \sum_{j \in \mathcal{N}_1(t)} f(x_j(t)-x_1(t)) (x_j(t)-x_i(t)).
	\]
By assumption, we have $x_k(t) \in (x_1(t)+\epsilon, x_1(t)+\gamma]$, so 
	\[
	x_1(t+1) - x_1(t) \ge \frac{1}{n} \inf_{x \in (\epsilon, \gamma]} x f(x).
	\]

Case (iii):  We consider two sub-cases.  First, suppose the neighbors of agent $1$ form a clique with agent $1$ that is disjoint from rest of the system.  That is, for $j \in \mathcal{N}_1$ we have $\mathcal{N}_j = \mathcal{N}_1$.  In that case, $|x_i(t) - x_j(t)| < \epsilon$ for all $i,j \in \mathcal{N}_1$.  In this case, the update in \eqref{eq:GenHKupdate} is identical for all $i,j \in \mathcal{N}$ and $x_1(t+1) = x_i(t+1)$ for all $i \in \mathcal{N}_1$.

In the second sub-case, let $k \in \mc{N}_1$ be the farthest neighbor of agent $1$ and note that $x_k(t) - x_1(t) \le \epsilon$.  Suppose there is a $k' \in \mc{N}_k$ such that $k' \notin \mc{N}_k$.  In that case, agent $k'$ is at least $(\gamma-\epsilon)$ away from agent $k$, otherwise agent $k'$ would have been a neighbor of agent $1$.  Therefore  
	\[
	f(x_{k'}(t)-x_k(t)) (x_{k'}(t) - x_k(t)) \ge f( \gamma-\epsilon ) (\gamma-\epsilon).
	\]
Hence at  this time step
an increase in opinion of $k$ caused by agent $k'$ is lower bounded by $\frac{1}{n} \inf_{x \in (\epsilon, \gamma]} x f(x)$. This is due to the fact that $k$ can have at most $n-1$ neighbors and $xf(x)$ is increasing
with $x$ for $x \in [0,\gamma]$. 

Therefore at time $t+1$, we have
	\[
	x_k(t+1) - x_1(t+1) \le \frac{1}{n}f( \gamma-\epsilon ) (\gamma-\epsilon).
	\]
Hence in the next time step
$t+2$
	\[
	x_1(t+2) - x_1(t+1) \ge \frac{1}{n^2}f( \gamma-\epsilon ) (\gamma-\epsilon).
	\]
Hence we observe that at any $t$ either $x_1(t)$ converges in one step (case (iii), first sub-case) or it increases
by at least $\frac{1}{n^2}\inf_{x \in (\epsilon, \gamma]} x f(x)$ every two time steps. 
Because $f$ is non-increasing and $f(\gamma)>0$, 
$\inf_{x \in (\epsilon, \gamma]} x f(x)$ is a positive constant. 

Note that agent $1$ converges to an opinion no more (less) than $M(0)$ ($m(0)$), 
as $M(t)$ ($m(t)$) is non-increasing (non-decreasing).
Because the initial spread of the opinions $M(0)-m(0)$ ($=x_n(0)-x_1(0)$) is finite, the convergence time for the agent $1$ is $O(n^2)$.

Once agent $1$ converges all the other agents that have not converged yet must be outside $\gamma$ neighborhood of agent $1$. Hence all these agents can be thought as a different generalized HK dynamics starting at that time and we can reduce the problem to a smaller set of agents, repeating the same argument for the minimum.
There are at most $n$ such reductions that we may have to consider, since at each step at least one agent converges.  Thus a (possibly loose) upper bound on the convergence time of the dynamics is $\text{O}(n^3)$.

The proof of the lower bound is similar to \cite{tsitsiklisHK09}.
\end{proof}

We can improve the dependence on $n$ in the upper bound on the convergence time at the expense of a dependence on the neighborhood radius $\gamma$.

\begin{corollary}
Under the assumptions of Theorem \ref{thm:UBconvTime}, the upper bound on the convergence time of the generalized HK dynamics is $O(n^2/\gamma)$, i.e.,
$\exists k>0$ such that for all $n$ sufficiently large convergence happens within $\frac{k n^2}{\gamma}$.
\end{corollary}

\begin{proof}
Instead of identifying the agents based with their indices, we can identify them with their locations. 
In previous theorem we show that $x_1(t)$ (and hence $m(t)$) converges in finite time, more precisely
in time $O(n^2)$. Convergence of $m(t)$ in a finite time $T_1$ implies that there is no agent $i$ with opinion $x_i(t) \in (\lim_{t\to \infty} m(t), \lim_{t\to \infty} m(t)+\gamma]$ for $t>T_1$. This in turn
with the monotonicity of the sequence $m(t)$ implies that there is no agent with opinions in 
$(m(0), m(0)+\gamma]$. Note that $T_1\leq k_2 n^2$ for some $k_2>0$ and this constant is 
independent of the initial configuration of opinions (as proved in previous theorem).

Thus we can say that after time $T_1$, the rest of the agents whose opinions are more than
$m(0)+\gamma$ at time $T$ can be considered as a new system starting with initial opinions
$x_i(T)$. The minimum opinion of this new system converges in additional time $T_2$ and any agent
whose opinion has not converged has an opinion more than $m(0)+2 \gamma$. 

Thus for a sequence of time $T_1, T_2, \cdots$ where $T_i\leq k_2 n^2$ for all $i$, all the agents
that have not converged by time $\sum_{i=1}^j T_i$ have opinion more than
$m(0)+ j \gamma$ at that time. Note that as $M(t)$ is a decreasing sequence, at no time an agent
can have opinion more than $M(0)$. Hence all agents must have converged by time
$\sum_{i=1}^l T_i$ %
 where $l=\lceil \frac{M(0)-m(0)}{\gamma}\rceil$. 

As $T_i\leq k_2 n^2$ for each $i$, the overall convergence time of the dynamics
is $\text{O}\left(\frac{(M(0)-m(0)+1) n^2}{\gamma}\right)$. Note that $M(0)-m(0)$ is finite
and does not scale with $n$ (in most cases it is assumed to be $1$), hence the order bound
follows.
\end{proof}

With regards to a qualitative characterization of the equilibrium point, we identify two important cases: bounded and unbounded influence. Bounded influence is the scenario where the influence function
$f$ has a finite support. This was an assumption in the transient analysis that we presented above. 
In case of unbounded influence, agents' influence function
has an unbounded support, i.e., $f(x)>0$ for all $x \in \mathbb{R}_+$.

In case of influence functions with finite supports we show that the convergence time is finite, more specifically in 
$O(n^2)$ steps the opinions converge to an equilibrium and do not change any more. The following Lemma 
characterizes the equilibrium in the bounded influence case.

\begin{lemma}
\label{lem:finiteInflEqui}
For any set of influence functions with finite supports there exists an initial opinion values of agents $\{x_i(0)\}$
for which the equilibrium opinion values $\{x^*_i\}$ are such that at least two of the agents have different equilibrium
opinion values.
\end{lemma}

\begin{proof}
We simply consider an initial condition where the opinions of the 
agents are ordered as $x_1(0) \le x_2(0) \cdots \le x_n(0)$ and there exists a $1 \le j<n$ such that
$x_{j+1}(0)-x_j(0)>\gamma$. This implies that agents $i\leq j$ and agents $i>j$ are non-interacting and the agents above and below evolve according to non-interacting dynamics. Note that as $M(t)$ (resp.~$m(t)$) is non-increasing (resp.~non-decreasing),
$x^*_i\leq x_j(0)$ for $i\leq j$ ($x^*_i\geq x_{j+1}(0)$) for $i>j$). This completes the proof.
\end{proof}

This means that in case of bounded support influences there are cases where equilibrium is not a consensus, i.e., there may
be groups of agents where agents within a group reach consensus but across groups there is no consensus. 

On the other hand, in case of influences with unbounded supports there is always a consensus of all agents. The following lemma states this formally.

\begin{lemma}
\label{lem:infiniteInflEqui}
For any unbounded influence opinion dynamics 
the equilibrium opinion values $\{x^*_i\}$ is such that $x^*_i=x^*_j$ for all $1\leq i \leq j\leq n$.
\end{lemma}
\begin{proof}
This lemma can be proved by contradiction. Let us assume that the lemma is not true.
Then at an equilibrium $\mathbf{x}^*$ there exists an agent (say $1$) such that 
$x^*_1\leq\min_{j \neq 1} x^*_j$ and $x^*_1 < \min_{j \in S} x^*_j$ for some
non-empty $S$. At the configuration
$\mathbf{x}^*$, by the update rule of the dynamics, opinion of 
agent $1$ is updated to $x^*_1+\frac{1}{n}\sum_{j} (x^*_j-x^*_1)f(x^*_j-x^*_1)$.
As $f(x)>0$ for $x>0$ (by unbounded support property), all  terms in the summation 
are non-negative and there is at least one strictly positive term. Hence the updated
value is strictly more than $x^*_1$ which contradicts the fact that $\mathbf{x}^*$ is an
equilibrium.
\end{proof}

It is also relatively simple to find examples of influence functions $f$ and initial opinions such that the convergence is only asymptotic.  That is, at any finite time $t$ there is an agent whose opinion changes to a new value at time $t+1$. The specific properties of the influence 
function have a strong impact on the nature of the equilibrium as well as on the convergence time, 
and the HK model of simple averaging and finite-time convergence is a special case.

\subsection{A cost-minimization perspective \label{sec:optim}}

In this section we interpret the HK dynamics as each agent choosing a new opinion that minimizes a cost function that depends on its opinion and its neighbors' opinions.  In the standard HK update rule each agent picks a new opinion that is the centroid of its neighbor's opinions.  For agent $i$, given its neighbor set and $\mc{N}_i(t)=\{j: |x_i(t) - x_j(t)|<\gamma\}$ and their opinions, the HK update rule $x_i(t+1) = \frac{1}{|\mc{N}_i(t)|} \sum_{j \in \mc{N}_i(t)} x_j(t)$ is the minimizer for the quadratic cost function 
	\[	
	\sum_{j \in \mc{N}_i(t)} (x-x_j(t))^2 = (x-x_i(t))^2 + \sum_{j \in \mc{N}_i(t)\backslash i} (x-x_j(t))^2.
	\]
In this section we examine more general cost functions instead of squared-distance.  In particular, we look at the following update rule:
	\begin{align}
	x_i(t+1) = \argmin_x \left\{ g_i\left(\left|x-x_i(t)\right|\right) + \sum_{j \in \mc{N}_i(t)\backslash i} h_i\left(\left|x-x_j(t)\right|\right) \right\}.
	\label{eq:gencost}
	\end{align}
For an agent $i$, we call the cost function that captures the cost incurred due to 
change of its own opinion as the \textit{inertial cost} and denote that by $g_i$. The cost incurred by agent $i$ for its difference in opinion with any of its neighbor is denoted by $h_i$ and we call it the \textit{disharmonic cost}. In a general scenario at each time $t$ an agent $i$ tries to minimize (greedily) the total cost by choosing the minimizer in \eqref{eq:gencost}.  The special case where $g_i(x)=h_i(x)=x^2$ for all $i$ is the case of original HK dynamics. As a first step we want to investigate convergence of this dynamics.

In a society there is always a cost incurred by a social agent when it differs in opinion
from its neighbors or it moves its own opinions to a new value. As discussed above HK dynamics model this scenario by assuming the cost incurred by an agent due to difference in opinion with its social neighbors as well as the cost to move its own opinion to be quadratic in difference (or change of opinion). In general this cost may be any arbitrary function. Moreover, the cost function that captures the difference with neighbors (disharmonic cost) and the cost
function that captures change of its own opinion (inertial cost) may be different. 

\begin{lemma}
\label{lem:nonQuadConv}
Suppose that for each $i$, the inertial cost $g_i$ and disharmonic cost $h_i$ are strictly increasing functions.  Then as $t \to \infty$, the dynamics given by \eqref{eq:gencost} converge to an equilibrium of opinions.
\end{lemma}
\begin{proof}
This proof is similar to the proof for that of the previously discussed generalization of the HK dynamics. Note that when $g_i$, $h_i$ are strictly increasing for all $i$, the sequence $\tilde{x}(t)=\max_i x_i(t)$ is  a decreasing sequence and the sequence $\hat{x}(t) = \min_i x_i(t)$ is an increasing sequence.  The remainder of the argument follows similarly.
\end{proof}

The assumption that $g_i$ and $h_i$ are increasing for each $i$ means that within the set of neighbors of an agent $i$, larger disagreements cost more than small disagreements. While the opinions converge to an equilibrium, the convergence time is not finite as in the HK dynamics.  To see this, consider the example of two agents. In case of HK dynamics with quadratic cost if the two
agents are within the threshold distance $\epsilon$ of each other convergence happens in one time step. This is because the minimizer for both agents' costs is $\frac{x_1+x_2}{2}$.
It is not hard to check that this is also true when $g_1=g_2=h_1=h_2=h$ for a function $h$ whose derivative $h'$ is strictly increasing. On the other hand, if $g_1(x)=g_2(x)=x^4$ and $h_1(x)=h_2(x)=x^2$ then this is not true and the convergence is asymptotic, i.e., for any $t$, $x_1(t+1)-x_1(t)>0$ if $0<|x_1(0)-x_2(0)|<\epsilon$.

\section{Externally Influencing the dynamics}
Given a society of individuals whose opinions update according to a generalized HK dynamics, how can the society be influenced by a single node?  This problem of external influence is interesting from the perspective of viral marketing, election campaigns, public policies and other phenomena where strategic agents attempt to influence a crowd.  As opposed to the static optimization problems that we considered earlier in this paper, influencing public opinions is a dynamic control problem, and its solution depends on the allowable actions of the influencing agent. In this work we consider two cases with different control actions. In one case we consider strategically placing an external agent that can influence others according to the generalized HK dynamics, while in the other case we consider manipulating the utility an agent sees.

\subsection{Placing an external agent}
\label{subsec:greedy}
We first turn to a model in which agents follow the original HK dynamics in \eqref{eq:HKupdateRewrite}, but at each time step we may insert an additional agent whose opinion influences the other agents, but who does not themselves follow the update rule.  As earlier we denote the opinions of truthful agents at time $t$ by $x_i(t)$, where $1\leq i \leq n$.  At each time, we may introduce an additional agent with opinion $x_0(t)$.  The goal is to select $\{x_0(t) : t \ge 1\}$ to drive the equilibrium value of the opinions to a value larger than a target $\theta$ as fast as possible.  More formally,
given a target opinion $\theta$, we consider the following problem:
\begin{align}
T^{*} = \min_{x_0(t):t \ge 1} \left\{ T : \forall t \ge T \ \mbox{and} \  1\leq i \leq n \ x_i(t) \ge \theta \right\}.
\label{eq:extAgents}
\end{align}
First, we prove a lower bound (in an order sense) on $T$ which is independent of the agent-placement scheme. Then we propose a simple greedy scheme whose time matches the same order-bound.

\begin{lemma}
There exists a constant $c>0$ and an initial configuration of opinions such that for any agent placement scheme 
$T \ge c (n/\gamma)$ for all $n$ sufficiently large.  That is, $T = \Omega(n/\gamma)$.
\end{lemma}

\begin{proof}
Consider the initial condition where all agents have the same opinion, i.e., $x_1=x_2 \cdots = x_n$. By the law of the dynamics
opinions of these agents cannot be separated by placing an external agent, since all of the agents
have the same neighborhood. Thus this dynamics can be represented as a dynamics of a single agent $x'$ with the following update rule:
	\[
	x'(t+1)=\frac{1}{n+\mathbf{1}\left(|x_0(t)-x'(t)|\leq \gamma\right)}\left(n x'(t)+\mathbf{1}\left(|x_0(t)-x'(t)|\leq \gamma\right) x_0(t)\right).
	\]
The goal is to place agent $x_0(t)$ to solve the following problem
	\begin{align*}
	\min_{x_0(t):t \ge 1} \left\{ T : \forall t \ge T,  x'(t) \ge \theta \right\}
	\end{align*}

From the dynamics it is apparent that the increase in $x'(t)$ at time $t$ does not depend on $x'(t)$ and only depends on 
$x_0(t)-x'(t)$.  The maximum possible increase corresponds to choosing $x_0(t) = x'(t) + \gamma$ at each time $t$.  In this case,
	\[
	x'(t+1) = x'(t) + \frac{\gamma}{n+1}.
	\]
So $x'(t)$ increases by at most $\frac{\gamma}{n+1}$, making the time to pass $\theta$ at least $\Omega(n/\gamma)$.
\end{proof}

In the general case, we analyze a scheme which we call \textit{greedy-recursive} for placing the external agent $x_0(t)$.  
At $t=0$, define
	\[
	k^*=\min\{i\geq 1: |x_{n-i}(0)-x_{n-i+1}(0)| > \gamma\}.
	\]
Then agents $n-k^* \leq i \leq n$ have no interaction with rest of the agents.  Our method attempts to greedily ``herd'' groups of agents across the threshold through a divide-and-conquer strategy.  At time $t$ we place the agent $0$ at a position $x_0(t)=x_{n-k^*}(t)+\gamma$ until $x_{n-k^*}(t)\geq \theta$. Then recursively apply the same strategy for the remaining $n-k^*$ agents.

\begin{lemma}
For the greedy-recursive scheme of placing an external agent, for $n$ sufficiently large and finite $\theta$, the time at which all agents have opinion greater than $\theta$ is upper-bounded by $c' \frac{n}{\gamma}$ for a constant $c'>0$.  That is, $T^* = O(n)$ in \eqref{eq:extAgents}.
\end{lemma}

\begin{proof}
Agents $n-k^* \leq i \leq n$ have no interaction with the rest of the agents. Hence, due to the greedy-recursive
scheme no agent $i<n-k^*$ is influenced by agent $0$ and they operate according to ordinary Hegselmann-Krause
dynamics.

For the agents $i\geq n-k^*$, note that, for any placement of agent $0$ the order of opinions are not 
disturbed. This follows from the order preservation in HK dynamics.
Hence it is sufficient to consider the time when agent $i=n-k^*$ crosses $\theta$. Note that due to the
greedy-recursive placement of the external agent at any time the agent $n-k^*$ increases by
at least $\frac{\gamma}{k^*}$. Hence this group of agents crosses $\theta$ within $\frac{(\theta-x_{n-k^*}(0))k^*}{\gamma}$ steps.

By same steps, for the next group agents (total $n+k^*-1$), there exists a $k_1^*\le n+k^*-1$ such
that agent $n-k^*-k_1^*$ is $\gamma$ away from agent $n-k^*-k_1^*-1$. Hence by the same argument as
above due to greedy placement of external agent for agents $n-k^*-1$ to $n-k^*-k_1^*$, these
agents cross $\theta$ within $\frac{(\theta-x_{n-k^*-k_1^*})k_1^*}{\gamma}$. Similarly we get bounds for
remaining agents and so on. 

Note that $\theta-x_j(0) \le \theta - x_1(0)$ for all $j$ . Also note 
for any sequence $k^*, k_1^*, k_2^*, \cdots $ such that
agents $[n:n-k^*]$ are $\gamma$ away from $[n-k^*-1:n-k^*-k_1^*]$ and so on,
$k^*+k_1^*+k_2^*+\cdots \le n$. 

Thus by adding all the times taken by greedy placements to move all the isolated group of agents
beyond $\theta$ is no more than $\frac{(\theta-x_1(0))n}{\gamma}$. 
\end{proof}

\subsection{Incentive schemes}

A different manner in which to externally influence dynamics is to modify the update rule at each time. As discussed in Section \ref{sec:optim} the generalized HK dynamics can be thought as a social interaction where at each time  $t$ an agent greedily minimizes a cost due to inertia and opinion differences with neighbors. These cost 
functions shape the way the agent updates its opinion. 
One way to move opinions of the system faster towards a desired point is to offer incentives 
to the agents to adjust their inertial cost functions. For example, offering some economic or social 
reward for adopting an opinion would reduce the inertial cost of agents. As the disharmonic cost depends on the
society rather than the individual, it is not apparent how an agent can be manipulated to change its 
disharmonic cost. In this section we discuss the case where the external force can only change the inertial 
cost.

Let $G_i$ be the set of possible inertial functions for an agent $i$. 
An example $G_i$ may be the set of all functions of the form $g_i(x)=x+a, \ x \ge 0$, i.e.,
$G_i=\{x+a:a \ge 0\}$. For any agent $i$, the external force has to pay an incentive to convince $i$
to choose a particular inertial cost. We measure this cost by a function $r_i:G_i \to \mathbb{R}_+$. 

There is another cost that the external force incurs per agent which is associated with the opinion of 
the agent. As before, the external force's goal is to take opinion of every agent to a desired opinion $\theta$ and the external force incurs a cost per agent depending on how far the agent's opinion is from $\theta$. 
For an agent $i$ we denote the cost by $\phi_i(\theta-x_i(t))$ at time $t$. 

Given sufficient time and a good initial condition, an external force can choose policies to offer 
incentives to agents so that they
eventually reach the opinion $\theta$. In real situations, however, the external force may only have a finite time in which to achieve this, and may also be budget-constrained in the sense that the total amount spent in offering incentives is upper-bounded. Under these constraints, it may not be possible to bring the entire society to the
opinion $\theta$. We therefore model the external force's goal as minimizing the aggregate cost over the whole time window.

Denote the inertial cost function and corresponding incentive of agent $i$ by $g^t_i$ and $r^t_i$ 
respectively.  Given a budget constraint $\rho$ and time horizon $T$, the following is the problem that an external force solves for optimal operations:
\begin{align}
\min_{\{g^t_i\}} & \sum_{t=1}^T \sum_i \phi_i(\theta-x_i(t)) & \nonumber \\
\mbox{s.t.} & \qquad \sum_t \sum_i r_i^t \leq \rho \nonumber \\
&\qquad x_i(t+1) = \arg \min_x g_i(|x-x_i|) + \sum_{j \in \mc{N}_i(t)\setminus i} h_i(|x-x_j|). \label{eq:controlProb}
\end{align}
First, note that  to take an action at $t$ (here $g^t_i, r^t_i$) the 
external force can only observe variables (here $x_i(t)$) and actions up to time $t$. Because the dynamics are 
deterministic and we are interested in a finite time window, the entire strategy can be chosen off-line.  This turns the control problem into an optimization constrained by the operation of the HK dynamics.

To understand the nature of this control problem we first consider a simpler version of the 
problem (\ref{eq:controlProb}) for a single time slot, i.e., for time horizon $T=1$. We further simplify the problem by taking the set of inertial costs $G_i$ for agent $i$ to be the set of all quadratic functions: $G_i=\{a x^2: 0 \leq a < 1\}$.  We define the incentive function for agent $i$ to be $r_i=1-g_i(x)/x^2$ for a given $g_i$.  Thus for $g_i(x) = ax^2$ the incentive is $r_i = (1-a)$.  We will assume $h_i$ and $\phi$ are quadratic as well, so $h_i(x) = \phi(x) = x^2$
The general control problem in (\ref{eq:controlProb}) becomes the following:
\begin{align}
\min_{\{r_i\}} & \sum_i (\theta-x_i(2))^2 \nonumber \\
\mbox{s.t.} &\qquad \sum_i r_i \leq \rho \nonumber \\
&\qquad  x_i(2) = \arg \min_x (1-r_i) (x-x_i(1))^2 + \sum_{j \in \mc{N}_i \setminus i} (x-x_j(1))^2 \nonumber \\
&\qquad \leftrightarrow x_i(2) = \frac{1}{\mc{N}_i-r_i}\left(\sum_{j \in \mc{N}_i} x_j(1) - r_i x_i(1)\right)\label{eq:quadCont}
\end{align}

Without loss of generality, we can assume that the initial opinions are non-negative. This is because
all the cost functions depend on relative differences of opinions rather than the values of opinions.
As any function of the form $\frac{\alpha x}{C-x}$ is a convex function for $0\leq x < C$, we have that
$\frac{1}{\mc{N}_i-r_i}\left(\sum_{j \in \mc{N}_i} x_j(1) - r_i x_i(1)\right)$ is a convex function of $r_i$ for
$0\leq r_i \leq 1$. Thus the optimization problem (\ref{eq:quadCont}) is a convex optimization problem.
So we can write a Lagrangian relaxation to the optimization problem and check the Karush-Kuhn Tucker (KKT) conditions for optimality
to find the optimal solution:
\begin{align}
	F(\mbf{r}) &= \sum_{i=1}^{n} \left( \theta - \frac{1}{ |\mc{N}_i| - r_i } \left( - r_i x_i + \sum_{j \in \mc{N}_i} x_j \right) \right)^2 + \lambda \left( \sum_{i=1}^{n} r_i - \rho \right) 
\end{align}
where $\lambda\geq 0$ is the Lagrangian multiplier for the budget constraint.
KKT conditions imply that the subgradient set of the dual function 
should contain $\mathbf{0}$. Because the Lagrangian is 
differentiable, this is equivalent to checking that all the partial derivatives of Lagrangian
with respect to $r_i$s and $\lambda$ are $0$:
\begin{align}
\frac{\partial F}{\partial r_i} &= - 2 x_i \left( \theta - \frac{1}{ |\mc{N}_i| - r_i } \left( -r_i x_i + \sum_{j \in \mc{N}_i} x_j \right) \right) \frac{\partial}{\partial r_i} \frac{r_i}{|\mc{N}_i| + r_i} + \lambda \nonumber \\
& =  - 2 x_i \left( \theta - \frac{1}{ |\mc{N}_i| + r_i } \left( r_i x_i + \sum_{j \in \mc{N}_i} x_j \right) \right) \frac{|\mc{N}_i|}{\left(|\mc{N}_i|+r_i\right)^2} + \lambda \nonumber 
\end{align}

In the special case where the external force has unlimited capacity of altering other agents' cost functions, then we do not have
the budget constraint and hence
	\[
	\left( \theta - \frac{1}{ |\mc{N}_i| + r_i } \left( r_i x_i + \sum_{j \in \mc{N}_i} x_j \right) \right) = 0, 
	\]
which in turn implies that
	\[
	r_i(\theta - x_i) =  \sum_{j \in \mc{N}_i} x_j - |\mc{N}_i| \theta.
	\]

In the general case, where there is a budget constraints, the solution involves solving third-order polynomials; while analytic solutions do exist, they will yield incentives $\{r_i\}$ in terms of fractional powers of $\lambda$.  Applying the constraint $\sum_i r_i(\lambda)=\rho$ to solve for $\lambda$ is not possible in closed-form but numerical methods can yield a solution.  We leave a more detailed investigation of these solutions and their structure for future work.

The one-step greedy strategy here can be repeated to yield a greedy strategy for the original optimization problem with time horizon $T > 1$. Directly finding the optimal strategy is challenging because the dynamics yield non-convex constraints on the strategy.  It may be possible to reduce the search space or provide convex relaxations of the general problem.  This might give some insight into complex path-planning algorithms for large distributed  autonomous networks.

\section{Simulations}

In this paper we present several variants and extensions of the Hegselmann-Krause opinion dynamics.  In this section, we examine the empirical effect of changing the HK model for passive dynamics as well as our strategies for externally influencing a network of dynamically updating agents. These simulations are intended to  validate our theoretical results and provide further insights into how changing optimizations lead to different qualitative and quantitative behavior.

\subsection{Equilibrium behavior with non-uniform weights}

Our first  extension of the HK dynamics is to introduce a non-uniform weighting for the individual node updates. Nodes assign a higher weight to opinions 'close' to their own, resulting in an ``inertia'' that diminishes the impact of neighbors whose opinions are close to the boundary of the neighborhood.
We compare the trajectories of original HK dynamics
and that of generalized HK dynamics (non-uniform weights) with different weight functions $f$ (i.e., difference dependence of weights on opinion-difference).

\begin{figure*}[]
\centering
\subfigure[Original HK dynamics $f=1$.]{
\includegraphics[scale=0.35]{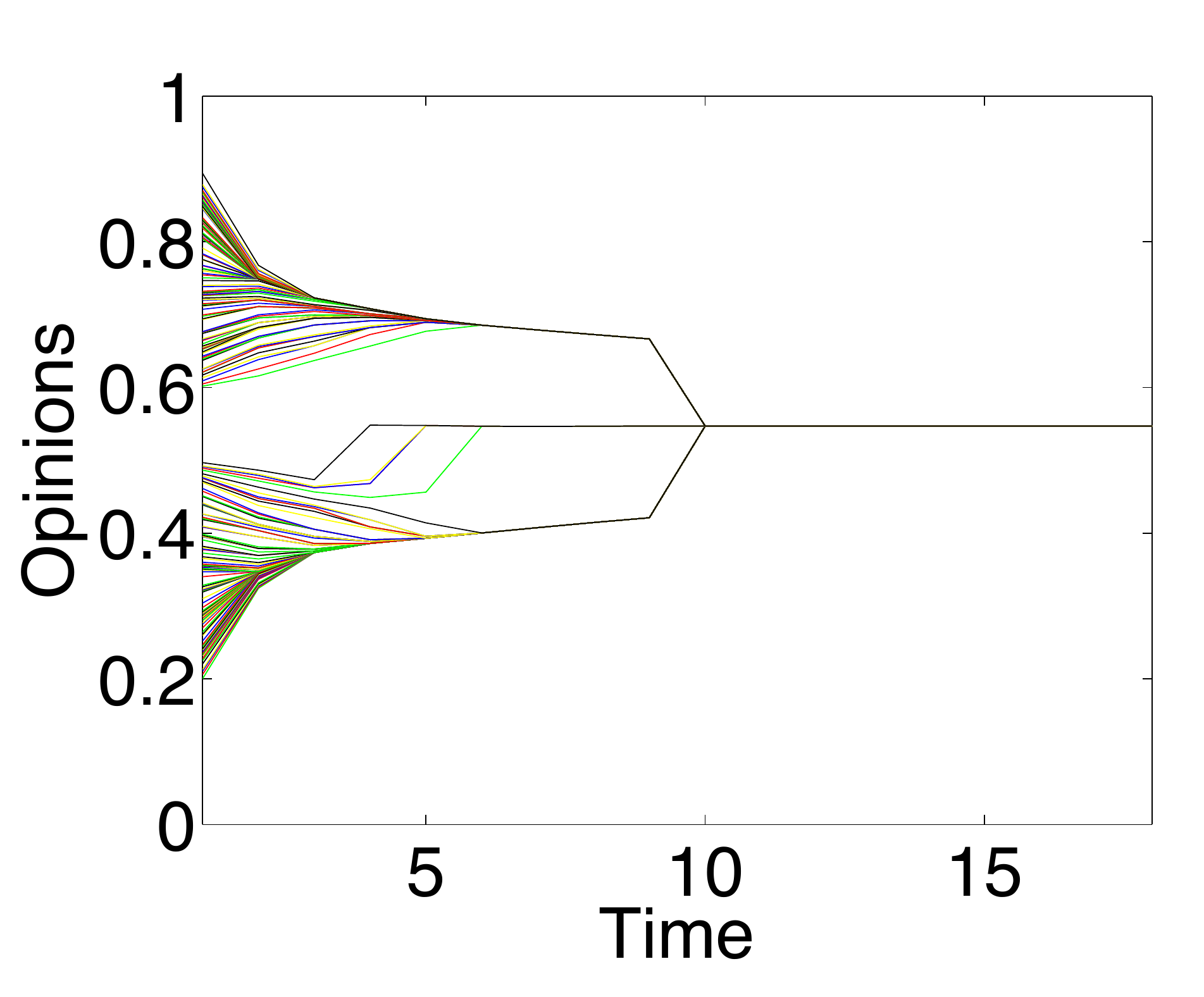}
\label{fig:origHK}}
\subfigure[Non-uniform HK with $f=\exp(-x^2)$.]{
\includegraphics[scale=0.35]{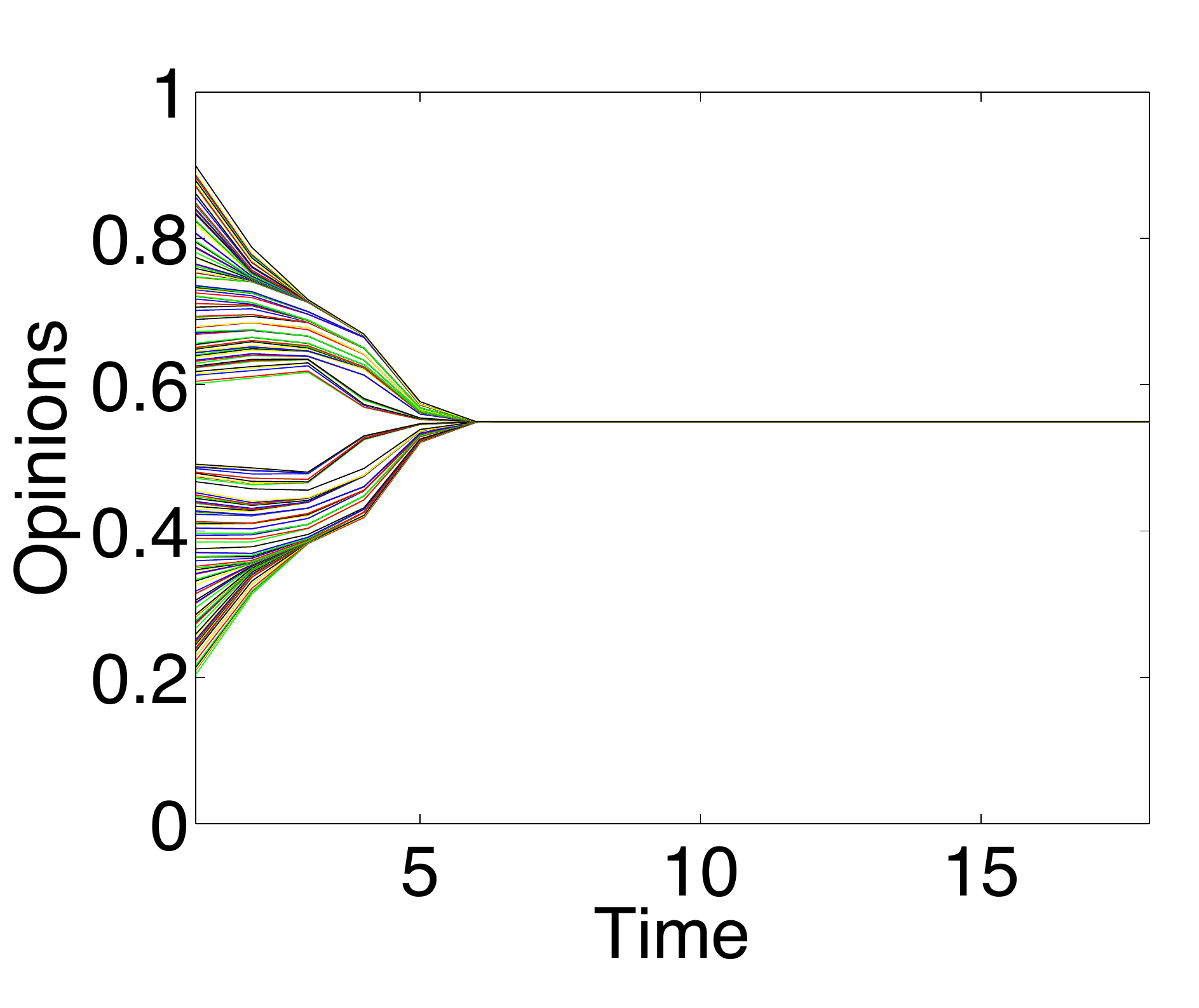}
\label{fig:nonUniformHKexpSq}}
\subfigure[Non-uniform HK with $f=\exp(-|x|)$.]{
\includegraphics[scale=0.35]{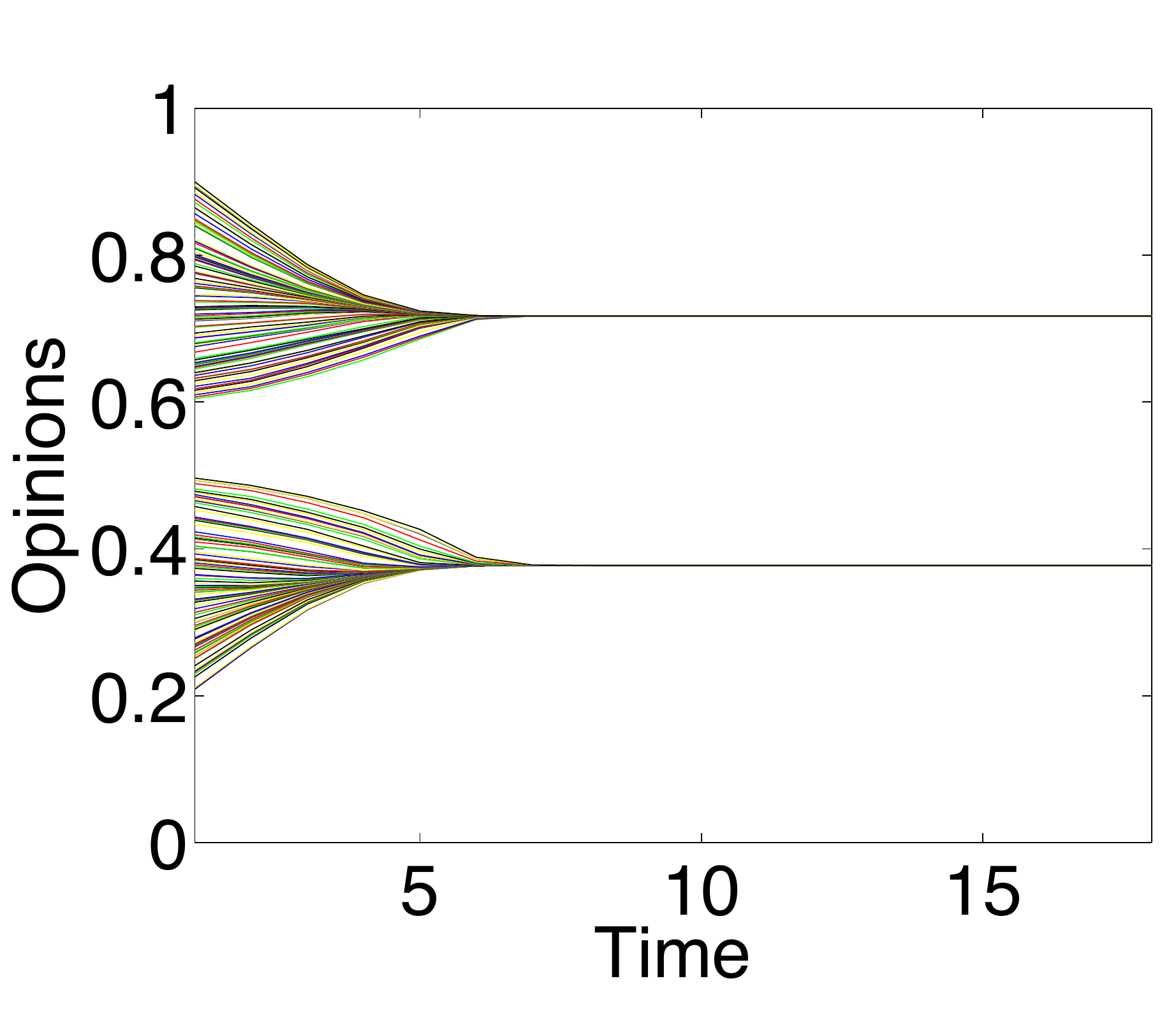}
\label{fig:nonUniformHKexp}}
\subfigure[Non-uniform HK with $f=\exp(-\sqrt{|x|})$.]{
\includegraphics[scale=0.35]{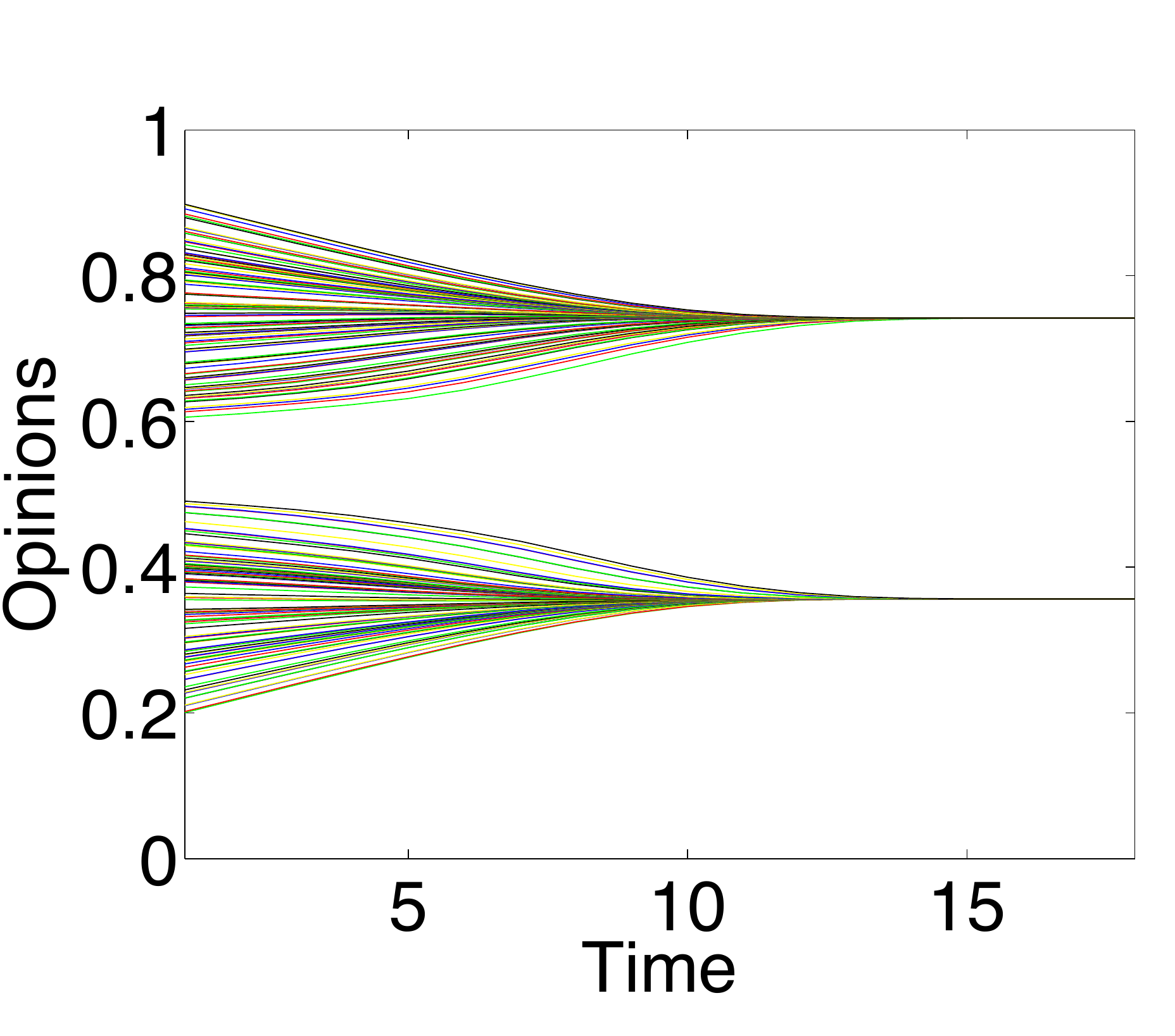}
\label{fig:nonUniformHKexpSqrt}}
\caption{Trajectory of generalized HK dynamics for different $f(\cdot)$ \label{fig:varyf}}
\end{figure*}

Figure \ref{fig:varyf} shows evolution of opinions for four different choices of the function $f(\cdot)$ in the generalized HK dynamics of \eqref{eq:GenHKupdate}.  We simulated $200$ agents with opinions in $[0,1]$ and show the evolution of their opinions according to the different choices of $f(\cdot)$ using neighborhoods of size $\gamma = 0.2$.  In our simulation, agents fall into two clusters -- interactions at the borders of these clusters is mediated by the decay of the function $f(\cdot)$.
Figure \ref{fig:origHK} plots the evolution under the original HK dynamics; in this case the opinions converge to a single
equilibrium.  Figure \ref{fig:nonUniformHKexpSq} shows the trajectory for non-uniform weighted dynamics with $f(x)=\exp(-x^2)$.
Here, the opinions converge to a single equilibrium, but the convergence is much faster than the original dynamics in Figure \ref{fig:origHK}.  Because the opinions are in $[0,1]$, the squared distance between agents with close opinions is much smaller than $1$.  Therefore $f(x)$ is close to $1$ for nearby agents but significantly smaller 
for farther agents. Nearby agents coalesce faster and coalesced masses of agents interact more strongly to attract other agents to form two major clusters.  We can see that in the original HK dynamics some agents agents move between the two main clusters (in the initial state) and slow
down the overall convergence rate. 

The results are quite different for other choices of $f(\cdot)$.  In Figure \ref{fig:nonUniformHKexp} we plot the trajectory for non-uniform weighted dynamics with $f(x)=\exp(-|x|)$. Note that
as opinions are in $[0,1]$, $\exp(-|x|)$ decays faster than $\exp(-x^2)$. This reduces exchange of opinions among farther agents and
hence the two clusters do not interact sufficiently strongly to cause a single coalescent point.  As a result, the opinions converge to two distinct equilibria.  The rate of convergence is slower because interaction among nearby agents is also weaker (since $|x|>|x|^2$ in $(0,1)$). 
A similar phenomenon is shown in Figure \ref{fig:nonUniformHKexpSqrt}, where we show the trajectory for dynamics with $f=\exp(-\sqrt{|x|})$.  Here the weights decay even faster with the separation between agents' opinions, and so the convergence rate is slowest. However, in all cases, although the convergence is slower with faster decaying $f$, in all cases convergence happens in finite time.

\subsection{Local optimization rules and dynamics}

We now turn to our generalized model in which agents perform local optimizations based on the opinions of their neighbors as in \eqref{eq:gencost}.  The HK dynamics is a special case of this dynamics with cost functions with all $g_i(\cdot)$ and $h_i(\cdot)$ being quadratic. We simulate dynamics with same initial
conditions, choosing $g_i(\cdot) = g(\cdot)$ and $g_i(\cdot) = h(\cdot)$ for all $1 \le i \le n$.  Our goal is to see the difference between the original dynamics and the dynamics with different optimizations.

\begin{figure*}[]
\centering
\subfigure[$g=x^2$ and $h=x^4$.]{
\includegraphics[scale=0.35]{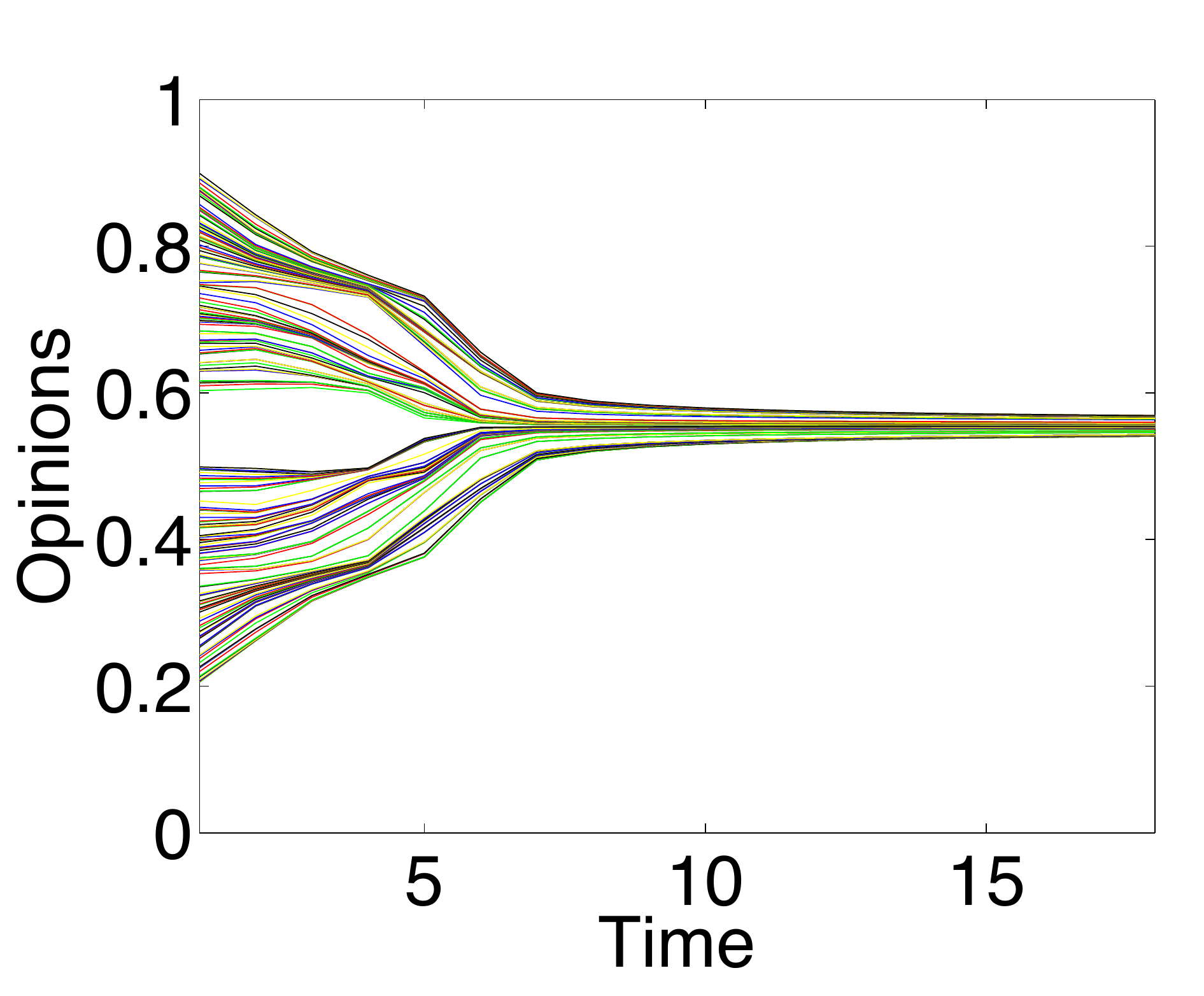}
\label{fig:X4nbrX2iner}}
\subfigure[$g=h=x^4$.]{
\includegraphics[scale=0.35]{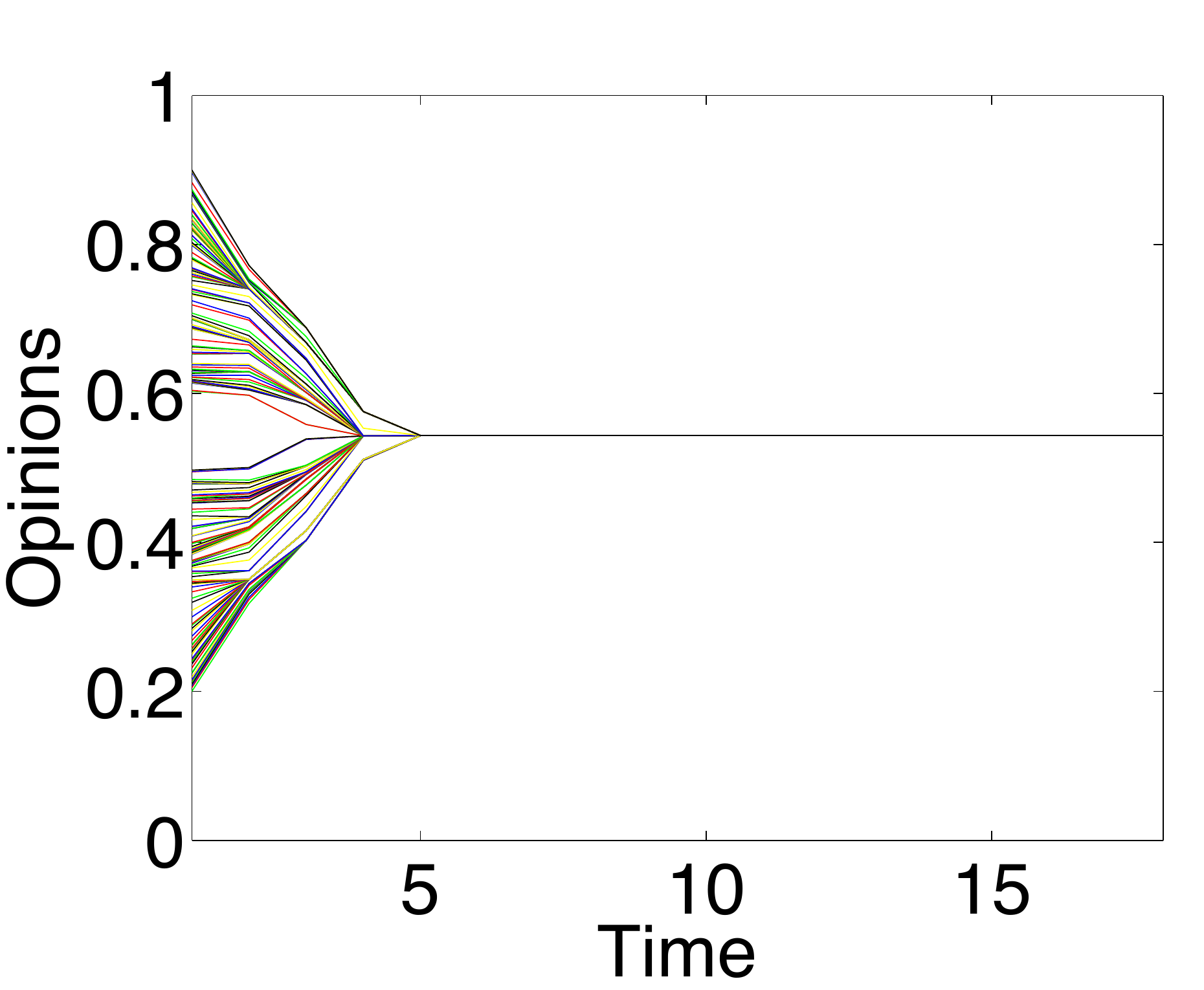}
\label{fig:X4Both}}
\caption{Trajectory of local optimization dynamics for different $g$ and $h$}
\end{figure*}

We simulate two different dynamics starting with the same initial condition as above, (i) $g(x)=x^2$ and $h(x)=x^4$ and (ii) $g(x)=h(x)=x^4$.
Note that in case (i) we have the same inertial cost as original HK dynamics but lower disharmonic cost because opinions are in $[0,1]$. For this dynamic, the agents' movements should be slower, as observed from the trajectory in Figure \ref{fig:X4nbrX2iner}.
As proved in our theoretical results, the convergence happens in asymptotic sense as opposed to in finite time. In case (ii) both the inertial and disharmonic costs are lower than the original HK dynamics. As inertial cost is a more binding factor in an agent's mobility, this dynamics converges faster, as shown in Figure \ref{fig:X4nbrX2iner}. Also note that in both of these dynamics opinions converge to a single equilibrium. 
In both the dynamics, though $h(x) =x^4$ causes agents to discount farther opinions while updating opinions, the discount 
factor is not as small as in the case of non-uniform dynamics with $f(x)=\exp(-x^2)$.

\begin{figure*}[]
\centering
\subfigure[Initial opinions are separated.]{
\includegraphics[scale=0.35]{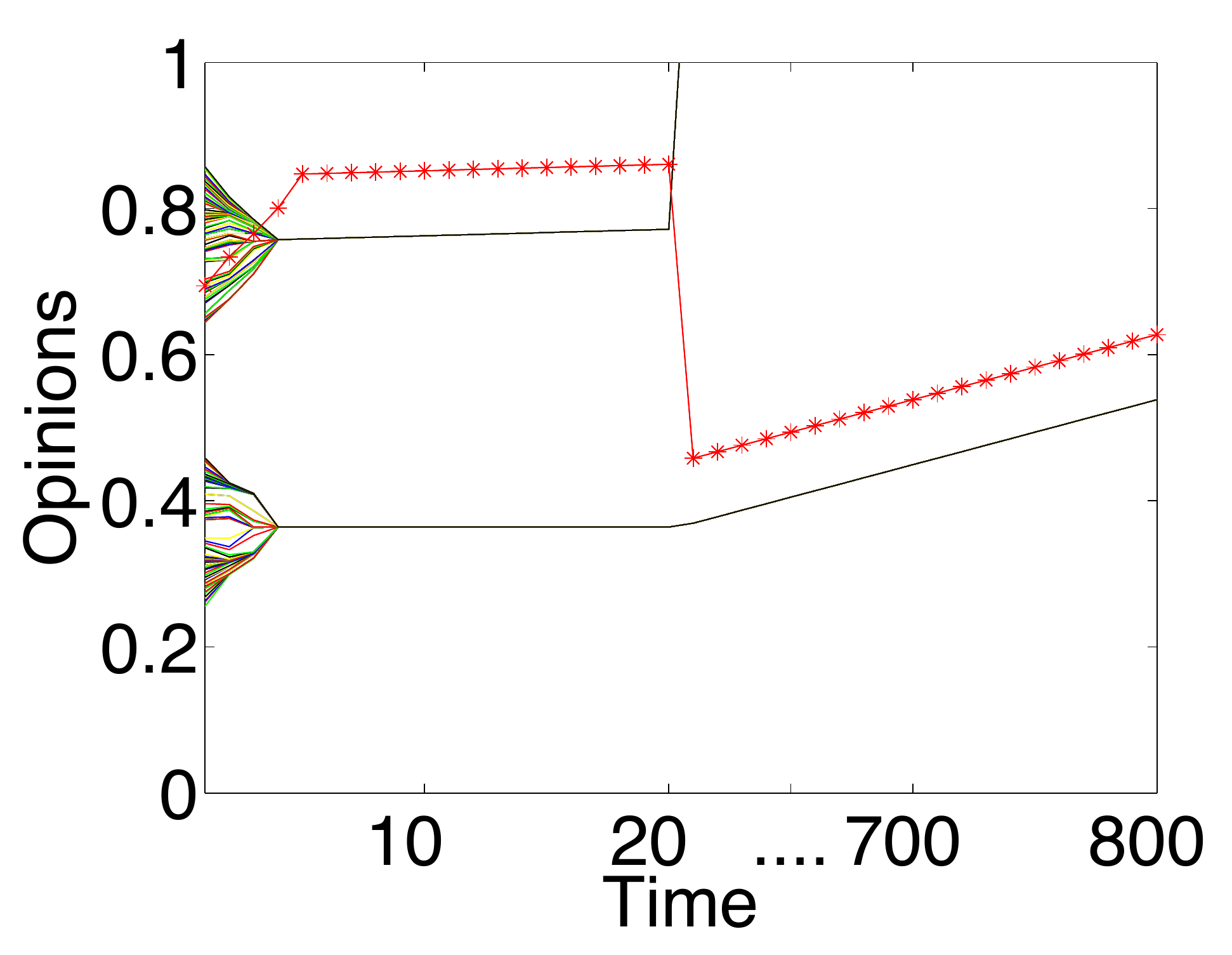}
\label{fig:Greedy1}}
\subfigure[Initial opinions are not separated.]{
\includegraphics[scale=0.35]{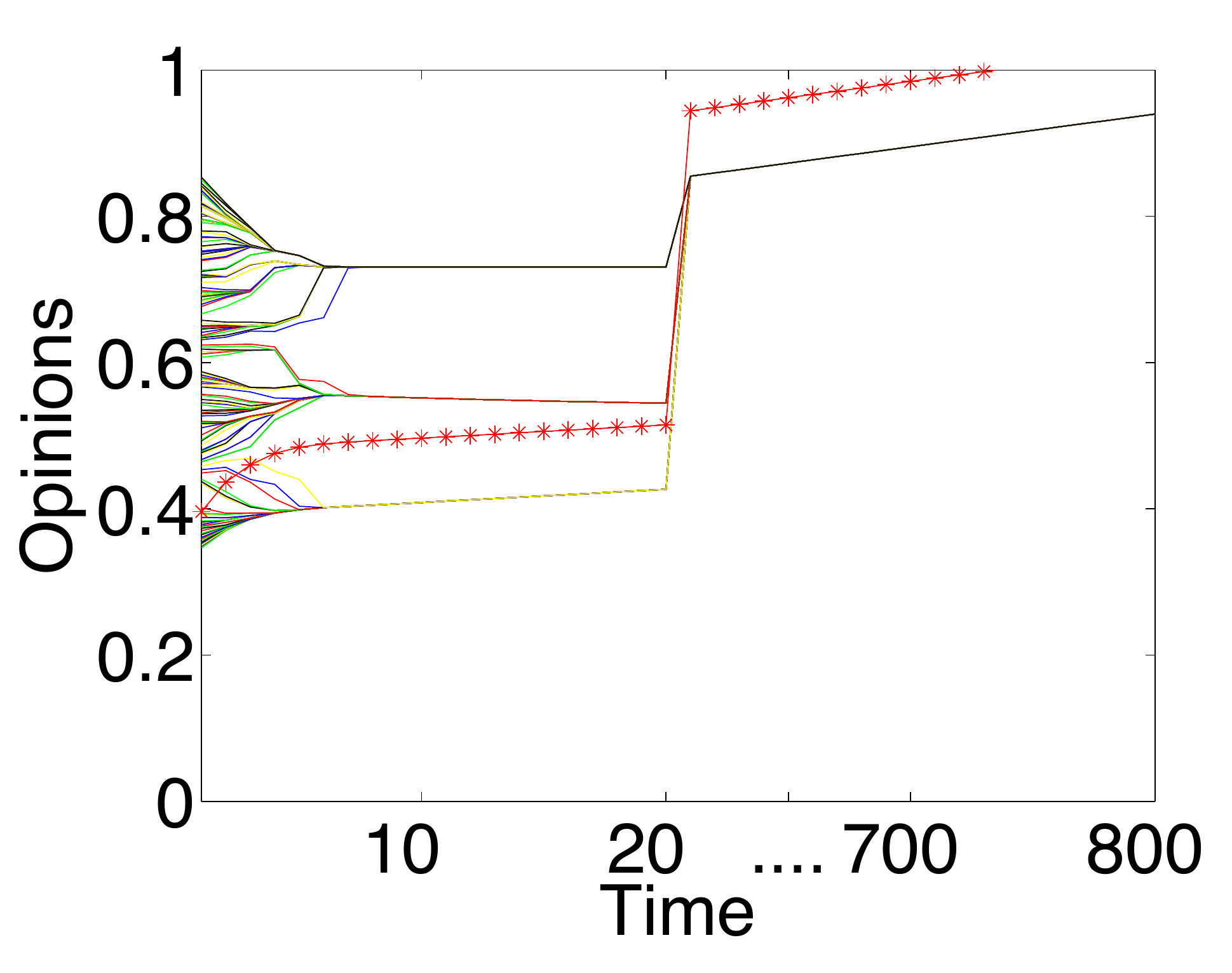}
\label{fig:Greedy2}}
\caption{Trajectory of agents under greedy algorithm. Red stars track placements of the external agent.}
\end{figure*}

In Section \ref{subsec:greedy} we show the order-optimality of greedy external agent placement. Here we want to understand how
the opinions actually change under this greedy placement and what the position of the external agent looks like. 
We use $\gamma=0.09$ and target opinion $\theta=1$.
We present trajectories of opinions for two different initial conditions. In Figure \ref{fig:Greedy1} the initial condition
is such that for each agent $i\notin\{1,200\}$, there are $j, j' \in N_i(0)$ such that $x_j>x_i$ and $x_{j'}<x_i$, i.e.,
there are no agents $u, v$ such that $x_u-x_v>\eta$ (threshold) and $\{i:x_i \in (x_v, x_u)\}=\emptyset$. On the other hand, in 
Fig. \ref{fig:Greedy2}, initially there is a pair of agents $u, v$ such that $x_u-x_v>\eta$ and $\{i:x_i \in (x_v, x_u)\}=\emptyset$.
That is is one case agents initially form a connected graph (in the sense that $(i,j)$ is an edge if $i \in N_j$ and $j \in N_i$)
where as in the other case the initial graph is not connected. We observe that in the case where agents are connected, the external
agent placement is monotonic and opinions of all agents merge before reaching the target value. In the disconnected case we
observe that the external agent placement is piecewise monotonic and in each such monotonic piece, the external agent takes 
one connected component (of initial graph) to the target opinion.

\section{Discussion}
In this paper, we generalize the Hegselmann-Krause (HK) model of opinion formation~\cite{hegselmannKrause02} in several directions. We first modify the weighting of neighbor opinions to be distance-dependent. This causes the dynamics to exhibit different qualitative behavior: there may be a loss in finite-time convergence, and the convergence times may depend on influence functions. We also interpret the HK update rule as a special case of agents greedily minimizing a cost associated to their differences in opinions.  The HK dynamics correspond to quadratic cost functions; and more generally, the cost may be understood as the sum of two quantities - an inertial cost representing the effort of an agent to change their opinion, and a disharmonic cost representing  the friction caused by having an opinion differing from their neighbors. In this more general setup, while we can prove convergence, such convergence need not happen in finite-time.  These results show the somewhat delicate nature of convergent behavior in HK dynamics, in that even slight perturbations in the model can prevent finite-time convergence.

Our second contribution is to understand influence modulation in such dynamics, where an external influencer seeks to drive the opinions of all agents past a target threshold. In the first variant, an external force can place an  agent at each time who can influence the dynamics without changing their own opinion.  We showed matching upper and lower bounds on the time to drive opinions past a target that are both $O(n)$.  In the second variant, we revisit the optimization framework and considered a case where an external force can alter the inertial costs of agents by offering incentives (at a cost).  This external force has finite time and finite budget to drive the agents past a target $\theta$, but this problem in general is non-convex and appears intractable. It is possible that a deeper understanding of  heuristic strategies may provide further insights into the structure of this problem; we plan to study this as future work.

Our work shows that there is a rich set of problems to explore beyond the original HK dynamics, especially in the context of control and optimization. The question more generally, is this: given a system evolving according to simple laws, how should we efficiently introduce external incentives that leverage the internal dynamics to achieve a certain goal?  While the social network model suggests this problem applies to marketing, politics, and propaganda, the real benefit may be in robotics, micro electro-mechanical systems (MEMS), and cellular and other biological systems.  Applying insights from the study of simple local dynamic rules may shed light on how to design scalable and efficient methods for controlling complex systems.

\bibliographystyle{IEEEtran}
\bibliography{social}

\end{document}